\documentclass[oneside,english,noend]{amsart}
\usepackage[T1]{fontenc}
\usepackage[latin9]{inputenc}
\usepackage{float}
\usepackage{amsthm}
\usepackage{amstext}
\usepackage{amssymb}

\makeatletter

\providecommand{\tabularnewline}{\\}
\floatstyle{ruled}
\newfloat{algorithm}{tbp}{loa}
\providecommand{\algorithmname}{Algorithm}
\floatname{algorithm}{\protect\algorithmname}

\numberwithin{equation}{section}
\numberwithin{figure}{section}
 \usepackage[noend]{morealgorithmic}
 \newcommand\numberedlinesswitch{0}
 \newcommand\enablelinenumbering{\renewcommand\numberedlinesswitch{1}}

\newtheorem{thm}{Theorem}
  \newtheorem{defn}[thm]{Definition}
  \newtheorem{prop}[thm]{Proposition}
\newenvironment{lyxlist}[1]
{\begin{list}{}
{\settowidth{\labelwidth}{#1}
 \setlength{\leftmargin}{\labelwidth}
 \addtolength{\leftmargin}{\labelsep}
 \renewcommand{\makelabel}[1]{##1\hfil}}}
{\end{list}}
  \newtheorem{cor}[thm]{Corollary}
 \newenvironment{algorithmicwrapper}{
   \begin{algorithmic}[\numberedlinesswitch]
 }{
   \end{algorithmic}
 }
 \newenvironment{inputsenv}{\INPUTS}{\ENDINPUTS}
 \newenvironment{outputsenv}{\OUTPUTS}{\ENDOUTPUTS}
 \newenvironment{bodyenv}{\BODY}{\ENDBODY}
 \newenvironment{whileenv}[1][condition?]{\WHILE{#1}}{\ENDWHILE}
 \newenvironment{ifenv}[1][condition?]{\IF{#1}}{\ENDIF}
 \newenvironment{forenv}[1][condition?]{\FOR{#1}}{\ENDFOR}
  \newtheorem{rem}[thm]{Remark}
  \newtheorem{lem}[thm]{Lemma}
  \newtheorem{fact}[thm]{Fact}

\usepackage{perry}

\enablelinenumbering
\definecolor{commentcolor}{rgb}{0.6, 0.6, 0.6}
\renewcommand\algorithmiccomment[1]{\textcolor{commentcolor}{ --- #1}}

\newcommand\sigvar{\ensuremath{\mathbf{F}}}
\newcommand\siglt{\ensuremath{\prec}}

\newcommand\siggt{\ensuremath{\succ}}
\newcommand\sigge{\ensuremath{\succeq}}

\newcommand\monset{\ensuremath{\mathbb{M}}}
\newcommand\sigset{\ensuremath{\mathbb{S}}}
\newcommand\monsigset{\ensuremath{\widehat{\mathbb{S}}}}
\newcommand\spol[2]{\ensuremath{S_{#1,#2}}}
\newcommand\sig{\ensuremath{\mathrm{sig}}}
\newcommand\minsig{\ensuremath{\mathcal{S}}}
\newcommand\nonminsig{(NM)}
\newcommand\rewritable{(RW)}
\newcommand\divH{(GM)}
\newcommand\suptopred{(GS)}
\newcommand\prisyz{(FM)}
\newcommand\faurew{(FR)}
\newcommand\faurewtoo{(FM')}
\newcommand\minarr{(AM)}
\newcommand\rewarr{(AR)}
\newcommand\rewggv{(GR')}
\newcommand\minmon{(MR)}

\newcommand\altsig{natural sig\-nature}
\newcommand\als{\textrm{sig}}

\newcommand\Syz{\ensuremath{\textit{Syz}}}

\newcommand\ggv{G$^2$V}

\AtBeginDocument{
  
}

\makeatother

\usepackage{babel}

\begin{document}

\title{Signature-based Algorithms to Compute Gr\"obner Bases}

\author{Christian Eder}

\address{Universit\"at Kaiserslautern}

\email{ederc@mathematik.uni-kl.de}

\author{John Perry}

\address{University of Southern Mississippi}

\email{john.perry@usm.edu}

\begin{abstract}
This paper describes a Buchberger-style algorithm to compute a Gr\"obner
basis of a polynomial ideal, allowing for a selection strategy based
on {}``signatures''. We explain how three recent algorithms can
be viewed as different strategies for the new algorithm, and how other
selection strategies can be formulated. We describe a fourth as an
example. We analyze the strategies both theoretically and empirically,
leading to some surprising results.
\end{abstract}
\maketitle
%ISSAC
%2010
%\category{I.1.2}{Symbolic and Algebraic Manipulation}{Algorithms}[Algebraic Algorithms]
%\category{F.2.1}{Analysis of Algorithms and Problem Complexity}{Numerical Algorithms and %Problems}[Computations on Polynomials]

\section{Introduction}

A fundamental tool of symbolic and algebraic computation is the method
of Gr\"obner bases. The first algorithm to compute a Gr\"obner basis
was introduced by Buchberger in 1965~\cite{Buchberger65}; subsequently,
the computer algebra community has developed a number of additional
algorithms, such as~\cite{BrickSlimgb_Journal,Fau99,GM88,MMT92}.
In recent years, a new genus of algorithm has emerged, exemplified
by F5 and \ggv~\cite{Fau02Corrected,GGV10}.

While they are presented as different algorithms, both use a property
called a {}``signature'' to control the computation and reduction
of $S$-poly\-nomials. While studying the two, we realized that they
could be viewed as variations in the selection strategy of a basic
algorithm common to both. A third recent algorithm of Arri likewise
fits this mold~\cite{ArriPerry2009}, so it seemed instructive to
formulate explicitly both the underlying structure and the three algorithms
as strategies for implementation. This provides a common theoretical
framework which allows a careful comparison, looking both at how the
criteria employed by the strategies are related, and at experimental
timings in a unified environment. For the latter, we employed both
interpreted code (via Sage~\cite{sage} and \textsc{Singular}~\cite{Singular312})
and compiled code (via \textsc{Singular}). We examined both the original
implementations and new implementations of the algorithms as {}``plugins''
to the common algorithm. For consistency, the algorithms should all
compute a reduced Gr\"obner basis incrementally, so when the reader
sees {}``F5'', s/he should understand {}``F5C''~\cite{EderPerry09}.

Section~\ref{sec:Background} reviews basic notation and concepts,
adapting different sources which vary considerably in notation~\cite{ArriPerry2009,EderPerry09,Fau02Corrected,GGV10,Zobnin_F5Journal}.
Theoretical contributions, which include the new {}``\als-redundant
criterion'', begin in Section~\ref{sub: sig-based algs} and continue
into Section~\ref{sec: past algorithms as strategies}. Section~\ref{sec: Comparison}
gets to the meat of comparing the strategies; both the analysis of
Section~\ref{sub: theoretical comparison} and the timings of Section~\ref{sec:Experimental-results}
produce surprising and unexpected results.

\section{Background}

\label{sec:Background}Let $i\in\mathbb{N}$, $\field$ a field, and
$R=\pring$. Throughout this paper, $F_{i}=\left(f_{1},\ldots,f_{i}\right)$
where each $f_{j}\in R$, and $I_{i}=\ideal{F_{i}}$ is the ideal
of $R$ generated by the elements of $F_{i}$. Fix a degree-compatible
ordering $<$ on the monoid $\monset$ of monomials of $x_{1},\ldots,x_{n}$;
for any $p\in R$, we denote $p$'s leading monomial by $\lm\left(p\right)$,
its leading coefficient by $\lc\left(p\right)$, and write $\lt\left(p\right)=\lc\left(p\right)\lm\left(p\right)$.
For brevity, we may denote $t_{p}=\lm\left(p\right)$, $c_{p}=\lc\left(p\right)$,
and $\lcm\left(t_{p},t_{q}\right)=t_{p,q}$.

Let $f_{i+1}\in R\backslash I_{i}$. We want an algorithm that,\emph{
}given a Gr\"obner basis $G_{i}$ of $I_{i}$,\emph{ }computes a
Gr\"obner basis of $I_{i+1}=\ideal{F_{i+1}}$, where $F_{i+1}=\left(f_{1},\ldots,f_{i+1}\right)$.

\subsection{The traditional approach}

Given $p,q\in R$, the \textbf{$S$-poly\-nomial} of $p$ and $q$
is
\[
\spol{p}{q}=\frac{t_{p,q}}{t_{p}}\cdot p-\frac{c_{p}}{c_{q}}\cdot\frac{t_{p,q}}{t_{q}}\cdot q.
\]
(It makes some things easier later if we multiply only $q$ by a field
element.) Further, given $p,r\in R$ and $G\subset R$, we say $p$
\textbf{reduces to $r$ modulo} $G$ if there exist $\Lambda_{1},\ldots,\Lambda_{\ell}\in\mathbb{N}$,
$t_{1},\ldots,t_{\ell}\in\monset$, $c_{1},\ldots,c_{\ell}\in\field$,
and $r_{0},\ldots,r_{\ell}\in R$ such that
\begin{itemize}
\item $r_{0}=p$ and $r_{\ell}=r$; and
\item for all $i=1,\ldots,\ell$,

\begin{itemize}
\item $r_{i}=r_{i-1}-c_{i}t_{i}g_{\Lambda_{i}}$, and
\item $\lm\left(r_{i}\right)<\lm\left(r_{i-1}\right)$.
\end{itemize}
\end{itemize}
\noindent Buchberger's algorithm computes the Gr\"obner basis $G$
of $\ideal{F_{i+1}}$ by computing $S$-poly\-nomials and reducing
them modulo the current value of $G$: initially the $S$-poly\-nomials
of $f_{j},f_{i+1}$ where $j=1,\ldots,i$; then, for any $S$-poly\-nomial
that reduces to nonzero $r$, also for the pairs $r,g$ where $g\in G$,
adding $r$ to the basis after determining new pairs. Termination
is guaranteed by Dickson's Lemma, since no polynomial is added to
$G_{i+1}$ unless it expands the $\monset$-submodule $\ideal{\lm\left(G_{i+1}\right)}$
of the Noetherian $\monset$-monomodule $\monset$~\cite{KR00}.

The following is fundamental to the traditional approach.
\begin{defn}
Let $s\in R$ and $G\subset R$, with $\#G=\ell$. We say that $s$
has a \textbf{standard representation with respect to $G$} if there
exist $h_{1},\ldots,h_{\ell}\in R$ such that $s=h_{1}g_{1}+\cdots+h_{\ell}g_{\ell}$
and for each $k=1,\ldots,\ell$ either $h_{k}=0$ or $\lm\left(h_{k}\right)\lm\left(g_{k}\right)\leq\lm\left(s\right)$.
We may also say that $\left(h_{1},\ldots,h_{\ell}\right)$ is a standard
representation of $s$ with respect to $G$.
\end{defn}
If $s$ reduces to $r$ modulo $G$, then it has a standard representation
modulo $G$; the converse, however, is often false.

\subsection{Signature-based strategies}

\label{sub: sig-based algs}

All algorithms to compute a Gr\"obner basis follow the basic blueprint
of Buchberger's algorithm, but recent algorithms introduce a new point
of view. We call these {}``signature-based strategies'', inasmuch
as their computations take {}``signatures'' into account.
\begin{defn}
Let $\sigvar_{1},\ldots,\sigvar_{m}$ be the canonical generators
of the free $R$-module $R^{m}$. Let $p\in\ideal{F_{i+1}}$, $j\in\mathbb{N}$
with $j\leq i+1$, and $h_{1},\ldots,h_{j}\in R$ such that $h_{j}\neq0$
and
\[
p=h_{1}f_{1}+\cdots h_{j}f_{j}.
\]
If $c=\lc\left(h_{j}\right)$ and $\tau=\lm\left(h_{j}\right)$, we
say that $c\tau\sigvar_{j}$ is \textbf{a \altsig} of $p$. Let $\sigset$
be the set of all \altsig s:
\[
\sigset=\left\{ c\tau\sigvar_{j}:c\in\field\backslash\left\{ 0\right\} ,\tau\in\monset,j=1,\ldots,m\right\} .
\]
We extend the ordering $<$ on $\mathbb{M}$ to a partial ordering
$\siglt$ on $\sigset$ in the following way: $c\sigma\sigvar_{j}\siglt d\tau\sigvar_{k}$
iff $j<k$, or $j=k$ and $\sigma<\tau$. If $j=k$, $\sigma=\tau$,
and $c=d$, we say that $c\sigma\sigvar_{j}=d\tau\sigvar_{k}$. If
$j=k$ and $\sigma=\tau$, we say that $c\sigma\sigvar_{j}$ and $d\tau\sigvar_{k}$
are \textbf{level}. We do not otherwise compare them.\end{defn}
\begin{prop}
The ordering $\siglt$ is a well-ordering on $\monsigset=\left\{ \tau\sigvar_{j}:\tau\in\mathbb{M},j=1,\ldots,m\right\} $.
Thus, for each $p\in\ideal{F_{i}}$, we can identify a unique, minimal,
monic \altsig.
\end{prop}
(Our {}``Propositions'' are either trivial or proved elsewhere.)
\begin{defn}
We call the unique minimal monic \altsig\  of $p\in R$ its \textbf{minimal
signature}. We denote the set of all \altsig s of $p$ by $\sig\left(p\right)$,
and the minimal signature of $p$ by $\minsig\left(p\right)$.\end{defn}
\begin{prop}
\label{pro: natural sigs}Let $p,q\in I_{i+1}$. Assume that $c\sigma\sigvar_{j}\in\sig\left(p\right)$
and $d\tau\sigvar_{k}\in\sig\left(q\right)$. Let $t\in\monset$.
We have
\begin{lyxlist}{(M)}
\item [{(A)}] $c\sigma\sigvar_{j}\in\sig\left(p\pm q\right)$ if $c\sigma\sigvar_{j}\siggt d\tau\sigvar_{k}$;
\item [{(B)}] $\left(c\pm d\right)\sigma\sigvar_{j}\in\sig\left(p\pm q\right)$
if $c\sigma\sigvar_{j}$ and $d\tau\sigvar_{k}$ are level and $c\pm d\neq0$;
and
\item [{(C)}] $ct\sigma\sigvar_{j}\in\sig\left(tp\right)$.
\end{lyxlist}
\end{prop}
\begin{cor}
\label{cor: sig-safe red preserves sig}Let $p,q\in I_{i+1}$. Assume
that $c\sigma\sigvar_{j}\in\sig\left(p\right)$ and $d\tau\sigvar_{k}\in\sig\left(q\right)$.
Suppose that there exist $a\in\field$ and $t\in\monset$ such that
$at\lt\left(q\right)=\lt\left(p\right)$.
\begin{lyxlist}{(M)}
\item [{(A)}] If $t\tau\sigvar_{k}\siglt\sigma\sigvar_{j}$, then $c\sigma\sigvar_{j}\in\sig\left(p-atq\right)$.
\item [{(B)}] If $t\tau\sigvar_{k}$, $\sigma\sigvar_{j}$ are level and
$ad\neq c$, then $\left(c-ad\right)\sigma\sigvar_{j}\in\sig\left(p-atq\right)$.
\end{lyxlist}
\end{cor}
\begin{defn}
\label{def: sig-safe reduction}If $\left(c\sigma\sigvar_{j},p\right),\left(d\tau\sigvar_{k},q\right)\in\sigset\times I_{i+1}$,
$a\in\field$, and $t\in\monset$ satisfy (A) or (B) of Corollary~\ref{cor: sig-safe red preserves sig},
we say that $p-atq$ is a \textbf{$\sigma$-reduction} \textbf{of
$p$ with respect to $q$}. Otherwise, $p-atq$ is \textbf{$\sigma$-unsafe}.
When it is clear from context that we mean $\sigma$-reduction for
appropriate $\sigma$, we refer simply to reduction.

We define a \textbf{$\sigma$-reduction of $p$ modulo $G$} analogously,
and say that it is \textbf{complete} when no reductions of the type
described in Corollary~\ref{cor: sig-safe red preserves sig} are
possible. It is \textbf{semi-complete} when reductions of type (B)
can be performed, but not reductions of type (A).
\end{defn}
We now adapt the notion of a standard representation to consider \altsig s.
\begin{defn}
\label{def: sig-representations}Let $G\subset\sigset\times I_{i+1}$
with $\#G=\ell$, and suppose that for each $\left(d\tau\sigvar_{j},g\right)\in G$
we have $\tau\sigvar_{j}=\minsig\left(g\right)$. We say that any
$\left(c\sigma\sigvar_{i+1},s\right)\in\sigset\times I_{i+1}$ has
a \textbf{standard representation with respect to $G$} (or \textbf{\als-representation},
or \textbf{$\sigma$-representation} for short) if $c\sigma\sigvar_{i+1}\in\sig\left(s\right)$
and there exist $h_{1},\ldots,h_{\ell}\in R$ such that $\left(h_{1},\ldots,h_{\ell}\right)$
is a standard representation of $s$ and each $h_{k}=0$ or $\sigma\sigvar_{i+1}\sigge\lm\left(h_{k}\right)\minsig\left(g_{k}\right)$.
\end{defn}
Just as a reduction of $p$ to zero modulo $G$ corresponds to a $\lm\left(p\right)$-representation
of $p$ with respect to $G$ in the traditional case, a $\sigma$-reduction
to zero modulo $G$ corresponds to a $\sigma$-representation with
respect to $G$.
\begin{thm}
\label{thm: asc sig =00003D> min sig}Let $c_{1}$, \ldots{}, $c_{\ell}\in\field$,
$\tau_{1},\ldots,\tau_{\ell}\in\monset$, and $p_{1},\ldots,p_{\ell}\in I_{i+1}\backslash\left\{ 0\right\} $
such that
\begin{align*}
G & =\left\{ \left(\sigma_{1}\sigvar_{j_{1}},g_{1}\right),\ldots,\left(\sigma_{M}\sigvar_{j_{M}},g_{M}\right)\right\} \\
 & \phantom{=}\cup\left\{ \left(\sigvar_{i+1},f_{i+1}\right),\left(d_{1}\tau_{1}\sigvar_{i+1},p_{1}\right),\ldots,\left(d_{\ell}\tau_{\ell}\sigvar_{i+1},p_{\ell}\right)\right\} ,
\end{align*}
$\sigma\sigvar_{j}=\minsig\left(g\right)$ for all $\left(c\sigma\sigvar_{j},g\right)\in G$,
and for all $\left(c\sigma\sigvar_{i+1},p\right)$, $\left(d\tau\sigvar_{j},q\right)\in G$
one of the following holds:
\begin{itemize}
\item $\left(\frac{t_{p,q}}{t_{p}}\cdot\sigma\right)\sigvar_{i+1}=\left(\frac{t_{p,q}}{t_{q}}\cdot\tau\right)\sigvar_{j}$,
or
\item $\left(\frac{t_{p,q}}{t_{p}}\cdot\sigma\right)\sigvar_{i+1}\siggt\left(\frac{t_{p,q}}{t_{q}}\cdot\tau\right)\sigvar_{j}$
and $\left(\frac{t_{p,q}}{t_{p}}\cdot\sigma\sigvar_{i+1},\spol{p}{q}\right)$
has a standard represen\-tation with respect to $G$.
\end{itemize}
Then $\widehat{G}=\left\{ g:\exists\left(\sigma\sigvar_{j},g\right)\in G\right\} $
is a Gr\"obner basis of $I_{i+1}$.\end{thm}
\begin{proof}
Recall from~\cite{BWK93} that $\widehat{G}$ is a Gr\"obner basis
of $I_{i+1}$ iff $\ideal{\widehat{G}}=I_{i+1}$ and $\spol{p}{q}$
has a standard representation with respect to $\widehat{G}$ for all
distinct $p,q\in\widehat{G}$. Since $G_{i}$ is a Gr\"obner basis
of $I_{i}$, we know that $\spol{p}{q}$ has a standard representation
for every $p,q\in G_{i}$, so it suffices to check only the pairs
$p,q\in\widehat{G}$ such that at least $p\not\in G_{i}$. For any
such pair where $\left(\frac{t_{p,q}}{t_{p}}\cdot\sigma\right)\sigvar_{i+1}\siggt\left(\frac{t_{p,q}}{t_{q}}\cdot\tau\right)\sigvar_{j}$,
by hypothesis $\left(\frac{t_{p,q}}{t_{p}}\cdot\sigma\sigvar_{i+1},\spol{p}{q}\right)$
has a standard representation with respect to $G$; by definition,
$\spol{p}{q}$ also has a standard representation with respect to
$\widehat{G}$.

Order the remaining $S$-poly\-nomials by ascending level signature,
and choose one such $S$-poly\-nomial $\spol{p}{q}$ such that $\left(c\sigma\sigvar_{i+1},p\right)$
and $\left(d\tau\sigvar_{i+1},q\right)$ are in $G$  and $\left(\frac{t_{p,q}}{t_{p}}\cdot\sigma\right)\sigvar_{i+1}=\left(\frac{t_{p,q}}{t_{q}}\cdot\tau\right)\sigvar_{i+1}$
is minimally level. Now, $s=\frac{t_{p,q}}{t_{p}}\cdot p-\frac{c}{d}\cdot\frac{t_{p,q}}{t_{q}}\cdot q$
has signature less than $\left(\frac{t_{p,q}}{t_{p}}\cdot\sigma\right)\sigvar_{i+1}$;
let $h_{1},\ldots,h_{j}\in R$ such that $s=\sum_{k=1}^{j}h_{k}f_{k}$
and $\lm\left(h_{j}\right)\sigvar_{j}=\minsig\left(s\right)$. Each
$\lm\left(h_{k}\right)\lm\left(f_{k}\right)\sigvar_{k}$ is smaller
than $\left(\frac{t_{p,q}}{t_{p}}\cdot\sigma\right)\sigvar_{i+1}$,
the mimimal level signature. By hypothesis, $S$-poly\-nomials corresponding
to top-cancellations among these $\lm\left(h_{k}\right)\lm\left(f_{k}\right)$
have \als-representations with respect to $G$. Thus, there exist
$H_{1},\ldots,H_{\#G}$ such that $s=\sum H_{k}g_{k}$, and for each
$k$, $H_{k}=0$ or $\lm\left(H_{k}\right)\lm\left(g_{k}\right)\leq\lm\left(s\right)$
and $\minsig\left(s\right)\sigge\lm\left(H_{k}\right)\minsig\left(g_{k}\right)$.
If $s=\spol{p}{q}$, we are done. Otherwise, $\lm\left(s\right)=t_{p,q}$,
implying $\lm\left(H_{\ell}\right)\lm\left(g_{\ell}\right)=t_{p,q}$
for some $\ell$. Thus, there exist $u\in\monset$ and $\left(\mu\sigvar_{j},g\right)\in G$
such that $u\lm\left(g\right)=t_{p,q}$, and $\left(u\mu\right)\sigvar_{j}\siglt\left(\frac{t_{p,q}}{t_{p}}\cdot\sigma\right)\sigvar_{i+1}$.
Since $\left(u\mu\right)\sigvar_{j}\siglt\left(\frac{t_{p,q}}{t_{p}}\cdot\sigma\right)\sigvar_{i+1}$,
by hypothesis $\spol{p}{g}$ and $\spol{g}{q}$ have \als-repre\-sen\-tations
with respect to $G$; since $\spol{p}{q}=\frac{t_{p,q}}{t_{p,g}}\spol{p}{g}+\frac{t_{p,q}}{t_{g,q}}\spol{g}{q}$,
so does $\spol{p}{q}$.

For the remaining $S$-poly\-nomials, proceed similarly by ascending
level signature. Top-cancellations of level signature will be smaller
than the working signature, and so will have been considered already.
\end{proof}
We need one more concept. It looks innocent, but is quite powerful.
We have not seen it elsewhere in the literature, but it is inspired
by other work; see Lemma~\ref{lem: criteria of arri, ggv suffice}.
\begin{defn}
\label{def: sig-redundant}Let $G\subset\sigset\times R$ and $\left(c\sigma\sigvar_{i+1},f\right)\in\sigset\times R$.
If there exists $\left(d\tau\sigvar_{i+1},g\right)\in G$ such that
$\tau\mid\sigma$ and $\lm\left(g\right)\mid\lm\left(f\right)$, then
$\left(\sigma\sigvar_{i+1},f\right)$ is \textbf{signature-redundant}
to $G$, or \textbf{\als-redundant} for short.
\end{defn}
The proof of Theorem~\ref{thm: asc sig =00003D> min sig} uses a
technique common to sig\-nature-based algorithms: proceed from smaller
to larger signature, reusing previous work to rewrite $S$-poly\-nomials.
This suggests an algorithm to compute a Gr\"obner basis; see Algorithm~\ref{alg: sig-based GB computation}.
\begin{algorithm}[t]
\caption{\label{alg: sig-based GB computation}Signature-based Gr\"obner basis
computation}

\begin{algorithmicwrapper}

\begin{inputsenv}

\STATE{$F_{i}\subset R$, such that $F_{i}$ is a Gr\"obner basis of $\ideal{F_{i}}$
and $f_{j}\not\in\ideal{f_{1},\ldots,f_{j-1}}$}

\STATE{$f_{i+1}\in R\backslash\ideal{F_{i}}$}

\end{inputsenv}
\begin{outputsenv}

\STATE{$G\subset\sigset\times R$ satisfying the hypothesis of Theorem \ref{thm: asc sig =00003D> min sig}}

\end{outputsenv}
\begin{bodyenv}

\STATE{Let $G=\left(\left(\sigvar_{1},f_{1}\right),\ldots,\left(\sigvar_{i},f_{i}\right),\left(\sigvar_{i+1},f_{i+1}\right)\right)$}

\STATE{\label{line: create initial P}Let $P=\left\{ \left(\frac{t_{f_{i+1},f}}{t_{f_{i+1}}}\sigvar_{i+1},f_{i+1},f\right):f\in F_{i}\right\} $}

\STATE{\label{line: init Syz}Initialize $\Syz$ \COMMENT{TBD}}
\begin{whileenv}[$P\neq\emptyset$\label{line: loop P}]

\STATE{\label{line: prune P}Prune $P$ using $\Syz$ \COMMENT{TBD}}

\STATE{\label{line: create S}Let $S=\left\{ \left(\sigma\sigvar_{i+1},p,q\right)\in P:\deg\sigma\mbox{ minimal}\right\} $}

\STATE{Let $P=P\backslash S$}
\begin{whileenv}[$S\neq\emptyset$\label{line: loop S}]

\STATE{\label{line: prune S}Prune $S$ using $\Syz,G$ \COMMENT{TBD}}

\STATE{\label{line: pick spoly}Let $\left(\sigma\sigvar_{i+1},p,q\right)\in S$
such that $\sigma\sigvar_{i+1}$ is minimal}

\STATE{Remove $\left(\sigma\sigvar_{i+1},p,q\right)$ from $S$}

\STATE{\label{line: sig-safe red of spoly}Let $\left(\sigma\sigvar_{i+1},r\right)$
be a semi-complete $\sigma$-reduction of $\spol{p}{q}$ mod $G$}

\STATE{\label{line: update Syz}Update $\Syz$ using $\left(\sigma\sigvar_{i+1},r\right)$
\COMMENT{TBD}}
\begin{ifenv}[$r\neq0$ and $\left(\sigma\sigvar_{i+1},r\right)$ not \als-redundant
to $G$\label{line: need to create crit pairs?}]

\begin{forenv}[$\left(\tau\sigvar_{i+1},g\right)\in G$ such that $g\neq0$ and $g$
not \als-redundant\label{line: create pairs w/nonzero polys}]

\begin{ifenv}[$\frac{t_{r,g}}{t_{r}}\cdot\sigma\neq\frac{t_{r,g}}{t_{g}}\cdot\tau$\label{line: don't create pairs for cancelling ves}]

\begin{ifenv}[$\frac{t_{r,g}}{t_{r}}\cdot\sigma>\frac{t_{r,g}}{t_{g}}\cdot\tau$\label{line: identify new vestige}]

\STATE{Let $\left(\mu\sigvar_{i+1},p,q\right)=\left(\frac{t_{r,g}}{t_{r}}\cdot\sigma\sigvar_{i+1},r,g\right)$}

\ELSE{}

\STATE{Let $\left(\mu\sigvar_{i+1},p,q\right)=\left(\frac{t_{r,g}}{t_{g}}\cdot\tau\sigvar_{i+1},g,r\right)$}

\end{ifenv}

\begin{ifenv}[$\deg\mu=d$]

\STATE{Add $\left(\mu\sigvar_{i+1},p,q\right)$ to $S$\label{line: higher-sig reduction}}

\ELSE{}

\STATE{Add $\left(\mu\sigvar_{i+1},p,q\right)$ to $P$\label{line: new crit pair}}

\end{ifenv}
\end{ifenv}
\end{forenv}
\end{ifenv}

\STATE{Append $\left(\sigma\sigvar_{i+1},r\right)$ to $G$}

\end{whileenv}
\end{whileenv}

\RETURN{$\left\{ \left(\sigma\sigvar_{i+1},g\right)\in G:g\neq0,\mbox{ not }\als\mbox{-redundant}\right\} $\label{line: return}}\end{bodyenv}
\end{algorithmicwrapper}
\end{algorithm}
 It follows the basic outline of Buchberger's algorithm, with several
exceptions.
\begin{itemize}
\item As we will show in Lemma~\ref{lem: sig(s) is signature unless sig-safe rep},
Algorithm~\ref{alg: sig-based GB computation} prepends a \altsig\
to each critical {}``pair'', and considers pairs by ascending \altsig\ 
(as in \cite{AlbrechtPerryF45,ArriPerry2009,GGV10,Zobnin_F5Journal}),
rather than by ascending lcm (as in \cite{Buchberger85,EderPerry09}).
\item Only $\sigma$-reductions of $\left(\sigma\sigvar_{i+1},r\right)$
are computed directly. We require semi-complete reduction, but complete
reduction implies this. If we perform a complete reduction and conclude
with $\left(c\sigma\sigvar_{i+1},r\right)$, we multiply $r$ by $c^{-1}$
to ensure $\sigma\sigvar_{i+1}\in\sig\left(r\right)$ for line~\ref{line: sig-safe red of spoly}.
Reductions that are $\sigma$-unsafe occur in line~\ref{line: higher-sig reduction}
via the generation of new critical pairs. Algorithm~\ref{alg: sig-based GB computation}
adds these to $S$ rather than $P$ to preserve the strategy of ascending
signature.
\item The \algorithmicif\ statement of line~\ref{line: need to create crit pairs?}
rejects not only zero polynomials, but \als-redundant polynomials
as well. This has a double effect in Lemma~\ref{lem: sig-cancelling sig implies drop}
and Theorem~\ref{thm: sig-based alg terms correctly}.
\item We adopted the following from F5, partly to illuminate the relationship
with this algorithm better. Algorithm~\ref{alg: sig-based GB computation}
is easily reformulated without them, in which case it begins to resemble
\ggv\ and Arri's algorithm.

\begin{itemize}
\item Critical pairs are oriented: any $\left(\sigma\sigvar_{i+1},p,q\right)\in P$
corresponds to $\spol{p}{q}=up-cvq$ where $\left(\tau\sigvar_{i+1},p\right)$,
$\left(\mu\sigvar_{j},q\right)\in G$ and $\sigma=u\tau\siggt v\mu$.
\item Line~\ref{line: create S} selects all pairs of minimal degree of
\altsig. With homogeneous polynomials and a degree-compatible ordering,
this selects all $S$-poly\-nomials of minimal degree. An inner loop
processes these by ascending signature.
\item Zero and \als-redundant polynomials are retained in the basis for
reasons that become clear later; however, line~\ref{line: create pairs w/nonzero polys}
prevents them from being used to compute new critical pairs, and line~\ref{line: return}
does not add them to the output.
\end{itemize}
\end{itemize}
\begin{rem}
For now, we define lines~\ref{line: init Syz},~\ref{line: prune P},~\ref{line: prune S},
and~\ref{line: update Syz} to do nothing, and discuss them in Section~\ref{sec: prune P, S}.
\end{rem}
We prove the correctness of Algorithm~\ref{alg: sig-based GB computation}
in several steps.
\begin{lem}
\label{lem: sig(s) is signature unless sig-safe rep}Suppose that
one of lines~\ref{line: create initial P},~\ref{line: higher-sig reduction},
or~\ref{line: new crit pair} creates~$\left(\sigma\sigvar_{i+1},p,q\right)$.
Write $s=\spol{p}{q}$; not only is $\sigma\sigvar_{i+1}\in\sig\left(s\right)$,
but $\sigma\sigvar_{i+1}=\minsig\left(s\right)$ unless $s$ already
has a \als-represen\-tation w.r.t.~$G$ when Algorithm~\ref{alg: sig-based GB computation}
would generate it.\end{lem}
\begin{proof}
That $\sigma\sigvar_{i+1}\in\sig\left(s\right)$ follows from Proposition~\ref{pro: natural sigs},
Corollary~\ref{cor: sig-safe red preserves sig}, and inspection
of the algorithm. For the second assertion, suppose that $\sigma\sigvar_{i+1}\neq\minsig\left(s\right)$.
Let $\tau\sigvar_{j}=\minsig\left(s\right)$; by definition, there
exist $h_{1},\ldots,h_{j}$ such that $j\leq i+1$, $s=h_{1}f_{1}+\cdots+h_{j}f_{j}$,
$h_{j}\neq0$ , and $\lm\left(h_{j}\right)=\tau$. Notice $\tau\sigvar_{j}\siglt\sigma\sigvar_{i+1}$.
Since the algorithm proceeds by ascending \altsig, top-cancellations
of smaller signature have been considered already. Hence, all top-cancellations
among the $h_{k}f_{k}$ would have \als-representations at the moment
Algorithm~\ref{alg: sig-based GB computation} would generate $s$.
We can therefore rewrite the top-cancellations repeatedly until we
conclude with a \als-representation of $s$.\end{proof}
\begin{lem}
\label{lem: sig-cancelling sig implies drop}Suppose that line~\ref{line: need to create crit pairs?}
prevents the algorithm from creating critical pairs using $\left(\sigma\sigvar_{i+1},r\right)$.
Then $r=0$, the corresponding $S$-poly\-nomials have \als-represen\-tations
already, or will after consideration of pairs queued in $P\cup S$.\end{lem}
\begin{proof}
If $r=0$, then we are done. Suppose $r\neq0$; by line~\ref{line: need to create crit pairs?},
there exist $\left(\tau\sigvar_{i+1},g\right)\in G$ such that $\tau\mid\sigma$
and $\lm\left(g\right)\mid\lm\left(r\right)$. Let $u\in\monset$
such that $u\lm\left(g\right)=\lm\left(r\right)$. If $u\tau<\sigma$,
then line~\ref{line: sig-safe red of spoly} did not perform a semi-complete
$\sigma$-reduction of $\spol{p}{q}$, a contradiction. Hence $u\tau\geq\sigma$.

Let $t\in\monset$ such that $u\tau\geq\sigma=t\tau$, so $\lm\left(r\right)=u\lm\left(g\right)\geq t\lm\left(g\right)$.
The signature $\mu\sigvar_{j}$ of $s=r-tg$ is smaller than $\sigma\sigvar_{i+1}$,
so Algorithm~\ref{alg: sig-based GB computation} has considered
top-cancellations of this and smaller \altsig. Hence, $\left(\mu\sigvar_{j},s\right)$
has a $\mu$-representation. Critical pairs have been generated for
$g$, so for any $\left(\zeta\sigvar_{i+1},q\right)\in G$, $\spol{r}{q}=ur-cvq=u\left(s+tg\right)-cvq$,
whose top-cancellations already have a \als-repre\-sentation or
will after consideration of pairs queued in $P\cup S$.\end{proof}
\begin{thm}
\label{thm: sig-based alg terms correctly}Algorithm~\ref{alg: sig-based GB computation}
terminates correctly.\end{thm}
\begin{proof}
\emph{Correctness:} If we show that the output of the algorithm satisfies
the hypothesis of Theorem~\ref{thm: asc sig =00003D> min sig}, then
we are done. By Lemma~\ref{lem: sig(s) is signature unless sig-safe rep}
and the strategy of ascending signature, we know for any $\left(\sigma\sigvar_{j},g\right)\in G$
that $g=0$ or $\sigma\sigvar_{j}=\minsig\left(g\right)$. The only
$S$-poly\-nomials for which the algorithm does not explicitly compute
\als-representations are those satisfying the criteria of the \algorithmicif\ statement
of line~\ref{line: need to create crit pairs?} and the criterion
of line~\ref{line: don't create pairs for cancelling ves}. The criterion
of line~\ref{line: need to create crit pairs?} is the hypothesis
of Lemma~\ref{lem: sig-cancelling sig implies drop}; with it and
the criterion of line~\ref{line: don't create pairs for cancelling ves},
we complete the hypothesis of Theorem~\ref{thm: asc sig =00003D> min sig}.

\emph{Termination:} Let $\monset'$ be the monoid of monomials in
$x_{1},\ldots,x_{2n}$; as with $\monset$, we can consider it to
be a Noetherian $\monset'$-monomodule. Any $\left(\sigma\sigvar_{i+1},r\right)$
added to $G$ with $r\neq0$ corresponds to an element of $\monset'$
via the bijection
\[
\left(\sigma,\lm\left(r\right)\right)=\left(\prod x_{i}^{\alpha_{i}},\prod x_{i}^{\beta_{i}}\right)\rightarrow\prod x_{i}^{\alpha_{i}}\prod x_{n+i}^{\beta_{i}}.
\]
Let $J$ be the $\monset'$-submodule generated by these elements
of $G$. Suppose the algorithm adds $\left(\sigma\sigvar_{i+1},r\right)$
to $G$ and $J$ does not expand; this implies that there exists $\left(\tau\sigvar_{i+1},g\right)\in G$
such that $\tau\mid\sigma$ and $\lm\left(g\right)\mid\lm\left(r\right)$.
Since $\left(\sigma\sigvar_{i+1},r\right)$ is \als-redundant, line~\ref{line: need to create crit pairs?}
prevents it from generating new pairs.

Hence, every time Algorithm~\ref{alg: sig-based GB computation}
adds $\left(\sigma\sigvar_{i+1},r\right)$ to $G$, either the submodule
$J$ expands, or the algorithm abstains from computing pairs. A submodule
of  $\monset'$ can expand only finitely many times, so the algorithm
can compute only finitely many pairs. Hence, the algorithm terminates.
\end{proof}
The following interesting result will prove useful; its criterion
is used in \cite{ArriPerry2009,GGV10} to prevent the generation of
new pairs.
\begin{lem}
\label{lem: criteria of arri, ggv suffice}To see if $\left(\sigma\sigvar_{i+1},r\right)$
is \als-redundant in Algorithm~\ref{alg: sig-based GB computation},
it suffices to check if there exist $\left(\tau\sigvar_{i+1},g\right)\in G$
and $t\in\monset$ such that $t\tau=\sigma$ and $t\lm\left(g\right)=\lm\left(r\right)$.\end{lem}
\begin{proof}
Assume that there exists $\left(\tau\sigvar_{i+1},g\right)\in G$
such that $\tau\mid\sigma$ and $\lm\left(g\right)\mid\lm\left(r\right)$.
Let $t,u\in\monset$ such that $t\tau=\sigma$ and $u\lm\left(g\right)=\lm\left(r\right)$.
If $u\tau<\sigma$, then line~\ref{line: sig-safe red of spoly}
did not compute a semi-complete $\sigma$-reduction of $\spol{p}{q}$,
a contradiction. If $u\tau>\sigma$, then $u>t$, so $\lm\left(r\right)=u\lm\left(g\right)>t\lm\left(g\right)$.
The signature of $r-tg$ is smaller than $\sigma$, so $r-tg$ has
a \als-representation with respect to $G$. In addition, $\lm\left(r-tg\right)=\lm\left(r\right)$;
by the definition of a \als-representation, there exist $\left(\mu\sigvar_{j},h\right)\in G$
and $u\in\monset$ such that $u\lm\left(h\right)=\lm\left(r-tg\right)=\lm\left(r\right)$
and $u\cdot\mu\sigvar_{j}$ is no greater than the signature of $r-tg$;
that is, $u\cdot\mu\sigvar_{j}\siglt\sigma\sigvar_{i+1}$. But then
line~\ref{line: sig-safe red of spoly} did not compute a semi-complete
$\sigma$-reduction of $\spol{p}{q}$, a contradiction.
\end{proof}

\subsection{Pruning $P$ and $S$}

\label{sec: prune P, S}This section does not propose any criteria
that have not appeared elsewhere; rather, it lays the groundwork for
showing how lines~\ref{line: init Syz},~\ref{line: prune P},~\ref{line: prune S},
and~\ref{line: update Syz} can use such criteria to improve the
efficiency of Algorithm~\ref{alg: sig-based GB computation}. The
general idea is:
\begin{itemize}
\item $\Syz$ will consist of a list of monomials corresponding to known
syzygies; i.e., if $t\in\Syz$, then $t\sigvar_{i+1}$ is a \altsig\
of a known syzygy.
\item Line~\ref{line: prune P} removes $\left(\sigma\sigvar_{i+1},p,q\right)$
from $P$ if there exists $t\in\Syz$ such that $t\mid\sigma$.
\item Line~\ref{line: prune S} does the same, and ensures that if $\left(\sigma\sigvar_{i+1},p,q\right)$,
$\left(\sigma\sigvar_{i+1},f,g\right)\in S$, then at most one of
these is retained.
\end{itemize}
Already, Lemma~\ref{lem: sig(s) is signature unless sig-safe rep}
suggests:
\begin{lyxlist}{(MM)}
\item [{\nonminsig}] Discard any $\left(\sigma\sigvar_{i+1},f,g\right)\in P\cup S$
if $\sigma\sigvar_{i+1}$ is not the minimal signature of $\spol{f}{g}$.\end{lyxlist}
\begin{prop}
\label{pro: faugere's criterion}Line~\ref{line: init Syz} can put
$\Syz=\left\{ \lm\left(g\right):g\in F_{i}\right\} $.
\end{prop}
For a proof, see Lemma 16 in~\cite{EderPerry09} (Faug\`ere's Criterion).
It is similar to the proof of the following criterion \cite{ArriPerry2009,GGV10}:
\begin{lem}
\label{lem: use the discovered non-trivial syzygies}If the result
of line~\ref{line: sig-safe red of spoly} is $\left(\sigma\sigvar_{i+1},r\right)$
with $r=0$, then line~\ref{line: update Syz} can add $\sigma$
to $\Syz$.\end{lem}
\begin{proof}
[(Sketch)]Suppose that line~\ref{line: sig-safe red of spoly} gives
$\left(\sigma\sigvar_{i+1},r\right)$ with $r=0$. Now, $r$ is the
$\sigma$-reduction of $s=\spol{p}{q}$ from line~\ref{line: sig-safe red of spoly},
and by Lemma~\ref{lem: sig(s) is signature unless sig-safe rep}
$\sigma\sigvar_{i+1}\in\sig\left(s\right)$. By definition, $\exists h_{1},\ldots,h_{i+1}\in R$
such that $s=\Sigma h_{k}f_{k}$ and $\lm\left(h_{i+1}\right)=\sigma$.
Since $r=0$, there exist $H_{1},\ldots,H_{\#G}$ such that $s=\Sigma H_{k}g_{k}$,
each $H_{k}=0$ or $\lm\left(H_{k}\right)\lm\left(g_{k}\right)\leq\lm\left(s\right)$,
and $\minsig\left(\Sigma H_{k}g_{k}\right)\siglt\sigma\sigvar_{i+1}$.
Hence $\Sigma h_{k}f_{k}-\Sigma H_{k}g_{k}=0$ and has \altsig\ $\sigma\sigvar_{i+1}$.
Suppose there exist $\left(\tau\sigvar_{i+1},p,q\right)\in P\cup S$
and $u\in\monset$ such that $u\sigma=\tau$; then $\spol{p}{q}=\spol{p}{q}-u\left(\Sigma h_{k}f_{k}-\Sigma H_{k}g_{k}\right)$
has signature smaller than $\tau$; now apply Lemma~\ref{lem: sig(s) is signature unless sig-safe rep}.
\end{proof}
Another criterion is implied by the following lemma.
\begin{lem}
\label{lem: rewritable criteria}Let $\tau\in\monset$, $B=\left\{ \left(\sigma_{j}\sigvar_{i+1},f_{j}\right)\in G:\sigma_{j}\mid\tau\right\} $.
We can choose any $\left(\sigma\sigvar_{i+1},f\right)\in B$ and discard
in line~\ref{line: prune S} any $\left(\tau\sigvar_{i+1},p,q\right)\in P$
if $p\neq f$, or if we compute $\left(\tau\sigvar_{i+1},f,g\right)$
where $p=f$ and $g\neq q$.\end{lem}
\begin{proof}
[(sketch)]Choose any $\left(\sigma\sigvar_{i+1},f\right)\in B$,
and let $\left(\mu\sigvar_{i+1},p\right),\left(\mu'\sigvar_{i+1},q\right)\in G$
such that $\left(\tau\sigvar_{i+1},p,q\right)\in P\cup S$. Let $t,u\in\monset$
such that $t\sigma=u\mu=\tau$. The signature of $up-tf$ is smaller
than $\tau$, so it has a standard representation with respect to
$G$; say $up-tf=\sum h_{k}g_{k}$. Then $\spol{p}{q}=up-vq=u\left(tf+\sum h_{k}g_{k}\right)-vq$.
All top-cancellations in this representation of $\spol{p}{q}$ are
of equal or smaller \altsig, so $\spol{p}{q}$ will have a standard
representation with respect to $G$ once line~\ref{line: pick spoly}
chooses $\sigma\sigvar_{i+1}\siggt\tau\sigvar_{i+1}$.
\end{proof}
Notice that Lemma~\ref{lem: rewritable criteria} requires only divisibility;
if there are no $\left(\tau\sigvar_{i+1},p,q\right)\in S$ such that
$p=f$, then we could discard all $\left(\tau\sigvar_{i+1},p,q\right)\in S$.
We thus have a {}``rewritable'' criterion:
\begin{lyxlist}{(RW)}
\item [{\rewritable}] For any $\tau\sigvar_{i+1}\in\monsigset$ select
$\left(\sigma\sigvar_{i+1},f\right)\in G$ such that $\sigma\mid\tau$;
discard any $\left(\tau\sigvar_{i+1},p,q\right)$ if $p\neq f$, or
if $p=f$ and we retain another $\left(\tau\sigvar_{i+1},f,g\right)\in S$
where $q\neq g$.
\end{lyxlist}
For simplicity's sake, we assume that we apply \rewritable\ only
in line~\ref{line: prune S}, but \nonminsig\ both there and in
line~\ref{line: prune P}.

\section{Known strategies}

\label{sec: past algorithms as strategies}This section sketches briefly
how Arri's algorithm, \ggv, and F5 can be viewed as strategies for
Algorithm~\ref{alg: sig-based GB computation}, distinguished by:
\begin{enumerate}
\item whether reduction is complete or semi-complete; and
\item how they prune $P$ and $S$.
\end{enumerate}
Space restrictions prevent us from going too far into each algorithm's
workings, or proving in detail the characterization of each as a strategy
for Algorithm~\ref{alg: sig-based GB computation}. However, the
reader can verify this by inspecting the relevant papers.

\subsection{Arri's algorithm}

Algorithm~\ref{alg: sig-based GB computation} is very close to Arri's
algorithm, which uses semi-complete reduction. Although \cite{ArriPerry2009}
presents this algorithm in non-incremental fashion, with a more general
way to choose the signatures, we consider it incrementally, with the
definition of signature as given here.

The algorithm maintains a list $G$ similar to that of Algorithm~\ref{alg: sig-based GB computation},
and discards $\spol{f}{g}$ if $f,g$ do not satisfy a definition
of a {}``normal pair''. This differs from the definition in~\cite{Fau02Corrected}:
\begin{defn}
\label{def: arri's normal pair}Any $f,g\in I_{i+1}$ are a \textbf{normal
pair} if $\spol{f}{g}=uf-cvg$ and
\begin{itemize}
\item for any $\left(\sigma\sigvar_{i+1},p\right)\in\left\{ \left(\minsig\left(f\right),f\right),\left(\minsig\left(g\right),g\right)\right\} $
there does not exist $\left(\tau\sigvar_{i+1},q\right)\in G$ and
$t\in\monset$ such that $t\tau=\sigma$ and $t\lm\left(q\right)=\lm\left(p\right)$;
\item $\minsig\left(uf\right)=u\cdot\minsig\left(f\right)$ and $\minsig\left(vg\right)=v\cdot\minsig\left(g\right)$;
and
\item $\minsig\left(uf\right)\neq\minsig\left(vg\right)$.
\end{itemize}
\end{defn}
In addition to $G$, Arri's algorithm maintains a list $L$ of leading
monomials used to prune $P$ (there called $B$). These correspond
to known syzygies; whenever $s$ \als-reduces to zero, the monomial
part of its \altsig\ is added to $L$.

We can characterize this as a strategy for Algorithm~\ref{alg: sig-based GB computation}
in the following way. The first bullet of Definition~\ref{def: arri's normal pair}
implies the \als-redundant property (Lemma~\ref{lem: criteria of arri, ggv suffice}).
To implement the second bullet,~\cite{ArriPerry2009} counsels initializing
$L$ to $\left\{ \lm\left(f\right):f\in F_{i}\right\} $ and adding
$\sigma$ to $L$ if the $\sigma$-reduction of $r$ concludes with~0.
This implements Proposition~\ref{pro: faugere's criterion} and Lemma~\ref{lem: use the discovered non-trivial syzygies}.
In addition,~\cite{ArriPerry2009} points out that for any fixed
\altsig\ one should keep a polynomial of minimal leading monomial;
after all, $\sigma$-reduction occurs when the leading monomial decreases
and the \altsig\ is preserved. Thus, the algorithm discards any $S$-poly\-nomial
if another polynomial of the same \altsig\ has lower leading monomial.
This implements Lemma~\ref{lem: rewritable criteria}. So, Arri's
algorithm discards $\left(\sigma\sigvar_{i+1},p,q\right)$ if either
of the following holds:
\begin{lyxlist}{(AM)}
\item [{\minarr}] for some $g\in F_{i}$, $\lm\left(g\right)\mid\sigma$,
or for some $\left(\tau\sigvar_{i+1},r\right)\in G$, $\tau\mid\sigma$
and $r=0$; or
\item [{\rewarr}] there exist $\left(\tau\sigvar_{i+1},g\right)\in G$
and $t\in\monset$ such that $t\tau=\sigma$ and $\lm\left(tg\right)<\lm\left(\spol{p}{q}\right)$,
or there exist $\left(\tau\sigvar_{i+1},f,g\right)\in S\cup P$ and
$t\in\monset$ such that $t\tau=\sigma$ and $\lm\left(t\spol{f}{g}\right)<\lm\left(\spol{p}{q}\right)$.
\end{lyxlist}
Notice that \rewarr\ checks divisibility of $\sigma$, not equality.
\begin{prop}
\label{pro: arri computes nat sigs}Arri's algorithm implements Algorithm
\ref{alg: sig-based GB computation} with semi-complete reduction:
\minarr\ implements \nonminsig\ and \rewarr\ implements \rewritable.
\end{prop}

\subsection{\label{sub:GGV}\ggv}

Although \ggv\ can be used to compute the colon ideal, we consider
it only in the context of computing a Gr\"obner basis. Thus, we are
really looking at a special case of \ggv.

\ggv\ maintains two lists of polynomials, $U$ and $V$. The polynomials
of $V$ are the elements of the basis. The polynomials of $U$ are
paired with those of $V$ such that
\begin{itemize}
\item if $v_{j}\in F_{i}$, then $u_{j}=0$;
\item if $v_{j}=f_{i+1}$, then $u_{j}=1$; and
\item if $v_{j}=ct_{k}v_{k}\pm dt_{\ell}v_{\ell}$ for some $c,d\in\field$,
$t_{k},t_{\ell}\in\monset$, and $v_{k},v_{\ell}\in V$, then $u_{j}=ct_{k}u_{k}\pm dt_{\ell}u_{\ell}$.
\end{itemize}
In addition, $S$-poly\-nomials and reductions are computed in such
a way that $\lm\left(u_{j}\right)$ is invariant for all $j$: \ggv\
computes $v_{j}-ctv_{k}$ only if $\lm\left(u_{j}\right)>t\lm\left(u_{k}\right)$
or $\lm\left(u_{j}\right)=t\lm\left(u_{k}\right)$ but $\lc\left(u_{j}\right)\neq\lc\left(cu_{k}\right)$.
Thus, if $u_{j}\neq0$, then $u_{j}\sigvar_{i+1}\in\sig\left(v_{j}\right)$.

The algorithm maintains another list $H$ of monomials that is initialized
with the leading monomials of all $f\in F_{i}$, and expanded during
the course of the algorithm by adding $\lm\left(u_{j}\right)$ whenever
$v_{j}$ reduces to zero. It does not compute an $S$-poly\-nomial
for $v_{j}$ and $v_{k}$ if:
\begin{lyxlist}{(GM)}
\item [{\divH}] $\max\left(\frac{t_{v_{j},v_{k}}}{t_{v_{j}}}\cdot u_{j},\frac{t_{v_{j},v_{k}}}{t_{v_{k}}}\cdot u_{k}\right)$
is divisible by a $t\in H$.
\end{lyxlist}
In addition, if every possible reduction of $v_{j}$ is by some $v_{k}$
such that $\lt\left(v_{j}\right)=dt\lt\left(v_{k}\right)$ and $\lt\left(u_{j}\right)=dv\lt\left(u_{k}\right)$,
then $v_{j}$ is \textbf{super top-reducible}, and \cite{GGV10} abstains
from generating critical pairs for $\spol{p}{q}$ if:
\begin{lyxlist}{(GS)}
\item [{\suptopred}] either $p$ or $q$ is super top-reducible.
\end{lyxlist}
Criterion~\divH\ implements Proposition~\ref{pro: faugere's criterion}
and Lemma~\ref{lem: use the discovered non-trivial syzygies}, while
\suptopred\ implies the \als-redundant property (Lemma~\ref{lem: criteria of arri, ggv suffice}).
\ggv\ offers no implementation of \rewritable\ beyond, {}``store
only one {[}pair{]} for each distinct {[}\altsig{]}''. \emph{Which}
pair is left somewhat ambiguous, but we will see that the choice is
important.
\begin{prop}
\label{pro: ggv computes nat sigs}\ggv\ implements Algorithm~\ref{alg: sig-based GB computation}
with complete reduction: \divH\ implements \nonminsig.
\end{prop}

\subsection{F5}

\label{sub:F5}As explained in the introduction, we use the F5C variant
of F5~\cite{EderPerry09}. In fact, we actually use a simplified
version of F5; we describe the differences below.

F5 maintains several lists $G_{1},\ldots,G_{i+1}\subset\monsigset\times R$;
for $j=1,\ldots,i$, each $G_{j}$ is a Gr\"obner basis of $\ideal{f_{1},\ldots,f_{j}}$.
Whenever a $\left(\sigma\sigvar_{i+1},r\right)$ concludes $\sigma$-reduction,
$\left(\sigma,r\right)$ is added to the $\left(i+1\right)$-st list
in a list named $\textit{Rules}$.

F5 discards $\left(\sigma\sigvar_{i+1},p,q\right)$ if:
\begin{lyxlist}{(FM)}
\item [{\prisyz}] for some $g\in F_{i}$, $\lm\left(g\right)\mid\sigma$,
or
\item [{\faurew}] there exists $\left(\tau,g\right)\in\mathit{Rules}_{i+1}$,
not \als-redundant, such that $g$ was computed after $p$ and $\tau\mid\sigma$.
\end{lyxlist}
Notice that \faurew, like \rewarr, checks divisibility of $\sigma$,
not equality.
\begin{prop}
\label{pro: f5 is implementation of sigbased alg}The simplified F5
described here implements Algorithm~\ref{alg: sig-based GB computation}
with semi-complete reduction: \prisyz\ implements \nonminsig, and
\faurew\ implements \rewritable.
\end{prop}
As noted, the F5 described here is simpler than~\cite{Fau02Corrected},
where:
\begin{itemize}
\item $S$-poly\-nomials of minimal degree are not computed in any particular
order (but the code of~\cite{EderPerry09} proceeds by ascending
lcm rather than ascending \altsig);
\item if $\spol{p}{q}=up-cvq$, then Criteria \prisyz\ and \faurew\ are
applied not only to $up$ but to $vq$; in addition, for any potential
$\sigma$-reduction $r-tg$, the criteria are used to reject some
$tg$.
\end{itemize}
Omitting these does not represent a significant difference from the
original algorithm. In fact, descriptions of F5 by ascending signature
have been around for some time~(\cite{AlbrechtPerryF45,Zobnin_F5Journal});
the second bullet can be viewed an optimization that makes sense in
an F4-style implementation, such as~\cite{AlbrechtPerryF45}.

There is one significant difference: the original F5 would not check
for \als-redundant polynomials. In view of this, when we view F5
as a strategy for Algorithm~\ref{alg: sig-based GB computation},
we will include the \als-redundant criterion; but in Section~\ref{sec:Experimental-results}
we will look at F5 with and without this criterion.

That said, in Section~\ref{sec:Experimental-results} we will look
at both the simpler F5 described here, and an implementation of the
original.

\section{Comparison of the algorithms}

\label{sec: Comparison}The thrust of Section~\ref{sec: past algorithms as strategies}
was to show that Arri's algorithm, \ggv, and a simplified F5 can
be viewed as implementations of Algorithm~\ref{alg: sig-based GB computation}.
Section~\ref{sub: theoretical comparison}, by contrast, compares
the three algorithms as their authors originally defined them, using
a strictly logical comparison of which critical pairs are discarded,
without regard to timings. Nevertheless, we retain the notation of
Algorithm~\ref{alg: sig-based GB computation}.

Section~\ref{sec:Experimental-results} compares the three algorithms
both ways: as implementations to, or {}``plugins'' for, Algorithm~\ref{alg: sig-based GB computation},
and for \ggv\ and F5 as standalone implementations. The results between
the approaches (plugin vs. original) do not differ significantly,
but provide both surprising results and additional insights.

\subsection{\label{sub: theoretical comparison}Logical comparison of the algorithms}

In this section we consider carefully how the criteria of Arri's algorithm,
\ggv, and the simplified F5 overlap. We begin with complete vs.~semi-complete
reductions:
\begin{fact}
\label{fact: additional reductions for GGV }\ggv\ reduces using
some polynomials that F5 and Arri's algorithm do not.\end{fact}
\begin{proof}
This is because reductions are complete in \ggv, but only semi-complete
in F5 and Arri's algorithm.
\end{proof}
On the other hand:
\begin{fact}
\label{fact: addition reductions useless}In Algorithm~\ref{alg: sig-based GB computation},
there cannot exist a reduction of type (B) in Corollary~\ref{cor: sig-safe red preserves sig}
without a reduction of\emph{ }type (A).\end{fact}
\begin{proof}
Let $\left(\tau\sigvar_{i+1},r\right),\left(\sigma\sigvar_{i+1},g\right)\in G$
such that $\lc\left(g\right)=d$, $\lc\left(r\right)=c\neq d$, and
there exists $t\in\monset$ such that $\lm\left(r\right)=t\lm\left(g\right)$
and $\tau=t\sigma$. Then $r-tg$ has a \altsig\ smaller than $\tau\sigvar_{i+1}$;
since the algorithm proceeds by ascending \altsig, $r-tg$ has a
\als-representation. Since $\lc\left(g\right)\neq\lc\left(r\right)$,
we have $\lm\left(r-tg\right)=\lm\left(r\right)$, so the \als-representation
of $r-tg$ implies that there exist $u\in\monset$, $\left(\mu\sigvar_{i+1},f\right)\in G$
such that $u\lm\left(f\right)=\lm\left(r\right)$ and $u\mu<\tau$.
This is a reduction of type (A).
\end{proof}
Together, Facts~\ref{fact: additional reductions for GGV } and~\ref{fact: addition reductions useless}
mean that while \ggv\ can reduce a polynomial using polynomials that
F5 and Arri's algorithm cannot, they can still reduce it using other
polynomials.

The next observation regards non-trivial syzygies.
\begin{fact}
\label{fac: comparison of minsig}\divH\ and \minarr\ are equivalent.
In addition, some pairs $\left(\sigma\sigvar_{i+1},p,q\right)$ rejected
by \divH\  and \minarr, but not by \prisyz, are also rejected \faurew.\end{fact}
\begin{proof}
It is trivial that \divH\ and \minarr\ are equivalent. Inspection
shows that \prisyz\ $\Longrightarrow$~\divH, but not the converse.
Assume therefore that \divH\  rejects a critical pair $\left(\sigma\sigvar_{i+1},p,q\right)$
that \prisyz\ does not; this implies that $\sigma$ is divisible
by some $\tau\in\monset$, where $\left(\tau\sigvar_{i+1},r\right)$
was the result of a complete or semi-complete reduction, and $r=0$.

Fact~\ref{fact: addition reductions useless} implies that $r$ $\sigma$-reduces
to zero in F5 as well. This is recorded in $\textit{Rules}$ by appending
$\left(\tau,r\right)$ to $\textit{Rules}_{i+1}$. Since $r=0$, F5
generates no more critical pairs for it. So $p,q\neq r$. If $r$
was generated after $p$, then \texttt{Rewritten$\left(\sigma\right)\neq p$},
so \texttt{Rewritten?$\left(\sigma\sigvar_{i+1},p\right)$} would
return \texttt{True}. In this case, \faurew\ rejects $\left(\tau\sigvar_{i+1},p,q\right)$.
Hence at least some pairs $\left(\sigma\sigvar_{i+1},p,q\right)$
rejected by \divH\  but not by \prisyz\  are also rejected by \faurew.

On the other hand, suppose that F5 computed $p$ after $r$; then
\texttt{Rewritten$\left(\sigma\sigvar_{i+1}\right)\neq r$}. In fact,
it might return $p$, and \texttt{Rewritten?}$\left(\sigma\sigvar_{i+1},p\right)$
would return \texttt{False}, so that F5 would not reject $\left(\sigma\sigvar_{i+1},p,q\right)$,
whereas \divH\ would.
\end{proof}
On the other hand, Lemma~\ref{lem: use the discovered non-trivial syzygies}
implies that one could modify \prisyz\ to consider zero reductions
as well as trivial syzygies, as the other algorithms do. We try this
in the next section. Alternately, one could modify \faurew\ to scan
$\textit{Rules}_{i+1}$ for pointers to zero reductions, using them
to discard pairs before performing the usual \faurew\ criterion.
Either works, but the former would likely be more efficient in interpreted
code; see a related discussion in the following section.
\begin{fact}
\label{fact: additional discards for GGV}Some critical pairs computed
by F5 are not computed by \ggv\ and Arri's algorithm.\end{fact}
\begin{proof}
From Lemma~\ref{lem: criteria of arri, ggv suffice}, we know that
\ggv\ and Arri's algorithm do not compute critical pairs for \als-redundant
polynomials, whereas F5 does.\end{proof}
\begin{fact}
\label{fact: additional discards for F5}Some pairs $\left(\sigma\sigvar_{i+1},p,q\right)$
discarded by \faurew\  are not discarded by \ggv. Likewise, some
pairs $\left(\sigma\sigvar_{i+1},p,q\right)$ discarded by \rewarr\ 
are not discarded by \ggv.\end{fact}
\begin{proof}
As noted in Section~\ref{sub:GGV}, \ggv\ implements~\rewritable\
by checking for equal signatures only, whereas~\faurew\  and~\rewarr\
check for divisibility. As a consequence, \faurew\ and \rewarr\
can discard $\left(\sigma\sigvar_{i+1},p,q\right)$ because $\left(\tau\sigvar_{i+1},f\right)\in G$
and $\tau\mid\sigma$, even if $f$ generates no pairs of \altsig\
$\sigma$.
\end{proof}

\subsection{\label{sec:Experimental-results}Experimental results}

Although Facts~\ref{fac: comparison of minsig} and~\ref{fact: additional discards for GGV}
imply an advantage for \ggv\ and Arri's algorithm over F5, Fact~\ref{fact: additional discards for F5}
implies an advantage for F5 and Arri's algorithm over \ggv. We will
see that the latter advantage is more significant than the former.

We first look at some timings of the algorithms as plugins for Algorithm~\ref{alg: sig-based GB computation}.
To do this, we implemented Algorithm~\ref{alg: sig-based GB computation}
as a C++ class in the kernel of a developer version of \textsc{Singular}
3-1-2, then created descendant classes corresponding to the other
algorithms. This allowed us to implement complete reductions for \ggv\
and semi-complete reductions for F5 and Arri's algorithm without giving
either an otherwise unfair advantage. Using compiled code allows us
to avoid the overhead of an interpreter, but the code was otherwise
unoptimized, linking to Singular's polynomial arithmetic. Table~\ref{tab: singular plugins}
lists timings in seconds corresponding to this implementation.
\begin{table}[t]
\begin{centering}
\begin{tabular}{|c|c|c|c|c|}
\hline 
Test case & F5 & \ggv & Arri & Arri+\minmon\tabularnewline
\hline 
Katsura-9 & 14.98 & 17.63 & 18.25 & 20.95\tabularnewline
\hline 
Katsura-10 & 153.35 & 192.20 & 185.76 & 220.01\tabularnewline
\hline 
Eco-8 & 2.24{*} & 0.49 & 0.45 & 0.53\tabularnewline
\hline 
Eco-9 & 77.13{*} & 13.15 & 5.20 & 14.59\tabularnewline
\hline 
Schrans-Troost & 3.7 & 5.3 & 6.46 & 7.72\tabularnewline
\hline 
F744 & 19.35{*} & 26.86 & 7.37 & 27.77\tabularnewline
\hline 
Cyclic-7 & 7.00 & 33.85 & 8.82 & 40.54\tabularnewline
\hline 
Cyclic-8 & 7310 & 26242 & 17672 & >8h\tabularnewline
\hline 
\end{tabular}
\par\end{centering}

{*}See the discussion in the text.

\caption{\label{tab: singular plugins}Timings, in seconds, of compiled \textsc{Singular}
implementations of the strategies, implemented as plugins for Algorithm~\ref{alg: sig-based GB computation}.
Computed on a workstation with a 2.66GHz Intel Core~2 Duo P8800 and
4~GB RAM, running 64-bit Ubuntu~10.10. Base field is $\field_{32003}$;
ordering is degree reverse lexicographic.}
\end{table}
 We do not include timings for a bare-bones Algorithm~\ref{alg: sig-based GB computation};
having no criteria to prune $P$ or $S$, it is unbearably slow. In
Table~\ref{tab: cps_reduced, zero_reductions}, we count the number
of critical pairs reduced, along with the number of zero reductions.
\begin{table}[t]
\begin{centering}
\begin{tabular}{|c|c|c|c|}
\hline 
Test case & F5 & \ggv & Arri\tabularnewline
\hline 
Katsura-9 & 886;0 & 886;0 & 886;0\tabularnewline
\hline 
Katsura-10 & 1781;0 & 1781;0 & 1781;0\tabularnewline
\hline 
Eco-8 & 830;322{*} & 2012;57 & 694;57\tabularnewline
\hline 
Eco-9 & 2087;929{*} & 5794;120 & 1852;120\tabularnewline
\hline 
Schrans-Troost & 380;0 & 451;0 & 370;0\tabularnewline
\hline 
F744 & 1324;342{*} & 2145;169 & 1282;169\tabularnewline
\hline 
Cyclic-7 & 1063;44 & 3108;36 & 781;36\tabularnewline
\hline 
Cyclic-8 & 7066;244 & 24600;244 & 5320;244\tabularnewline
\hline 
\end{tabular}
\par\end{centering}

{*}See the discussion in the text.

\caption{\label{tab: cps_reduced, zero_reductions}Number of critical pairs
reduced by each strategy, followed by number of zero reductions. We
omit Arri+\minmon\ for space; it is comparable to \ggv.}
\end{table}

As an example, we considered a fourth strategy that is essentially
Arri's algorithm, but we replace \rewarr\ by
\begin{lyxlist}{(MM)}
\item [{\minmon}] there exist $\left(\tau\sigvar_{i+1},g\right)\in G$
and $t\in\monset$ such that $t\tau=\sigma$ and $g$ has fewer monomials
than $\spol{p}{q}$, in which case we consider $tg$ in place of $\spol{p}{q}$;
if $\left(\sigma,f,g\right)\in S$ then we may choose either $\left(\sigma,p,q\right)$
or $\left(\sigma,f,g\right)$ freely.
\end{lyxlist}
This implementation of \rewritable\ causes \emph{more} work than
necessary. The first polynomials generated tend to have the fewest
monomials, so \minmon\ selects these instead of later polynomials,
and so repeats many earlier reductions. The algorithm still computes
a Gr\"obner basis, but takes the scenic route.

In general, F5 terminated the most quickly, but there were exceptions
where Arri's algorithm did. This is explained by the discussion in
the proof of Fact~\ref{fac: comparison of minsig}: \faurew\ is
sometimes too aggressive, and does not notice some zero reductions.
We modified \prisyz\ to check for these first, and this modified
F5 terminates for Eco-8 (-9) in 0.38s (8.19s), and for F744 in 8.79s.
Regarding critical pairs, it computes 565 (1278) critical pairs for
Eco-8 (-9), of which 57 (120) reduce to zero; and 1151 critical pairs
for F744, of which 169 reduce to zero. This suggests that computing
$\lm\left(\spol{p}{q}\right)$ in order to check \rewarr\ is usually
too expensive for the benefit that we would expect, but in some cases
it may be worthwhile.

The results of Table~\ref{tab: singular plugins} surprised us,
in that it contradicts the unequivocal assertion of~\cite{GGV10}
that \ggv\ is {}``two to ten times faster'' than F5. Apparently,
this is because~\cite{GGV10} compared \emph{implementations} and
not \emph{algorithms}. Why is this problematic? Primarily it is due
to the use in~\cite{GGV10} of interpreted code. Some natural adjustments
to the implementation of F5 used in~\cite{EderPerry09,GGV10} are
needed merely to start making the two comparable. We tried the following:
\begin{itemize}
\item Formerly, the F5 implementation checked \prisyz\ using an interpreted
\texttt{for} loop, but the implementation of \ggv\ checked \divH\
using \textsc{Singular}'s \texttt{reduce()}, pushing the \texttt{for}
loop into compiled code (line~173, for example). We changed the F5
code to check \prisyz\ using \texttt{reduce()}.
\item As specified in \cite{Fau02Corrected},\texttt{ TopReduction} and
\texttt{CritPair} return after each reduction or creation of a critical
pair. This back-and-forth incurs a penalty; it is sensible to loop
within these functions, returning only when reduction is semi-complete
or all critical pairs have been considered. (The code for \ggv\ did
this already.)
\item In the code accompanying \cite{EderPerry09},\texttt{ Spol} computes
$S$-poly\-nomials by ascending lcm, but one could proceed by ascending
\altsig\ instead~\cite{AlbrechtPerryF45,Zobnin_F5Journal}. (The
code for \ggv\ did this already.)
\end{itemize}
Some other changes contributed a little; see Table~\ref{tab: singular library comparison}
for timings, and Section~\ref{sec: conclusion} for source code.
With this new implementation, \ggv\ sometimes outperforms F5, but
by a much smaller ratio than before. Tellingly, F5 outperforms \ggv\
handily for Cyclic-$n$. This is despite the persistence of at least
one major disadvantage: the code to check~\faurew\ still uses an
interpreted \texttt{for} loop. This imposes a penalty not only in
\texttt{Spol} but also in \texttt{IsReducible} (called \texttt{find\_reductor}
in the implementation); the interpreted \texttt{for} loop checking
\faurew\ is one reason \texttt{IsReducible} consumes about half the
time required to compute a Gr\"obner basis in some systems.

\begin{table}
\begin{centering}
\begin{tabular}{|c|c|c|c|c|}
\cline{5-5} 
\multicolumn{1}{c}{} & \multicolumn{1}{c}{} & \multicolumn{1}{c}{} &  & F5/\ggv\tabularnewline
\cline{1-4} 
Test case & F5 & \ggv & F5/\ggv &  in~\cite{GGV10}\tabularnewline
\hline 
Katsura-6 & 1.2 & 0.53 & 2.26 & 6.32\tabularnewline
\hline 
Katsura-7 & 12.2 & 6.7 & 1.82 & 4.91\tabularnewline
\hline 
Katsura-8 & 134.28 & 50.1 & 2.68 & 5.95\tabularnewline
\hline 
Schrans-Troost & 261.69 & 50.61 & 5.17 & 14.04\tabularnewline
\hline 
F633 & 16.19 & 2.72 & 5.95 & 14.50\tabularnewline
\hline 
Cyclic-6 & 7.1 & 7.3 & 0.97 & 3.90\tabularnewline
\hline 
Cyclic-7 & 693.3 & 962.5 & 0.72 & 3.12\tabularnewline
\hline 
\end{tabular}
\par\end{centering}

\caption{\label{tab: singular library comparison}Timings, in seconds, of \textsc{Singular}
libraries, \emph{not} implemented as plugins to Algorithm~\ref{alg: sig-based GB computation}.
Computed on a MacBook Pro with a 2.4GHz Intel Core~2 Duo and 4~GB
RAM, running OS X 10.6.5. Base field is $\field_{32003}$; ordering
is degree reverse lexicographic.}

\end{table}

For \ggv, on the other hand, we observed a disadvantage inherent
to the algorithm and not to the implementation: its implementation
of \rewritable\ checks only equality, not divisibility. Fact~\ref{fact: additional discards for F5}
is especially evident in Table~\ref{tab: cps_reduced, zero_reductions}!
Most $S$-poly\-nomials are subsequently discarded as super top-redu\-cible.
These discards are not reflected in Table~2 of~\cite{GGV10}, which
mistakenly implies \ggv\ computes fewer polynomials than F5. \emph{Time
lost in reduction is unrecoverable.}

A final note. When we first implemented \ggv\, we obtained much worse
timings than those of Table~\ref{tab: singular plugins}. When we
inspected the source code accompanying~\cite{GGV10}, we found that
the choice of which pair to store for a given signature is not arbitrary:
the implementation prefers the most recently computed polynomial that
can generate a given signature (lines 176--183). That is, it discards
$\left(\sigma\sigvar_{i+1},p,q\right)\in S$ if
\begin{lyxlist}{(GR)}
\item [{\rewggv}] there exists $\left(\sigma\sigvar_{i+1},f,g\right)\in P$
such that $p\neq f$ and $f$ was computed after $p$, or $p=f$ and
$q$ was computed after $g$.
\end{lyxlist}
Note the similarity with \faurew.

\section{Conclusion}

\label{sec: conclusion}This paper has described a common algorithm
for which F5, \ggv, and Arri's algorithm can be considered strategies,
implemented as {}``plugins''. Algorithm~\ref{alg: sig-based GB computation}
makes use of a new criterion to reject polynomials and to guarantee
termination, not only for itself, but for \ggv\ and Arri's algorithm
as well, through Lemma~\ref{lem: criteria of arri, ggv suffice}.
The matter is not so clear for F5, which seems to terminate all the
same; we have not yet determined if some mechanism in F5 implies the
\als-redundant criterion.

Both timings and logical analysis imply that claims in~\cite{GGV10}
that \ggv\ is {}``two to ten times faster than F5'' are based on
a flawed comparison of implementations, rather than criteria. Indeed,
the inherent advantages appear to lie with F5 and Arri's algorithm:
in cases where F5 is not the most efficient, Arri's is, and modifying
\prisyz\ to check for non-trivial syzygies usually turns the scales
back in F5's favor:
\begin{lyxlist}{(FR')}
\item [{\faurewtoo}] $\max\left(\frac{t_{p,q}}{t_{p}}\cdot\sigma,\frac{t_{p,q}}{t_{q}}\cdot\tau\right)$
is divisible by $\lm\left(g\right)$ for some $g\in F_{i}$, or by
$\mu$ for some $\left(\mu\sigvar_{i+1},0\right)\in G$.
\end{lyxlist}
Source code developed for Table~\ref{tab: singular library comparison},
along with a demonstration version of the algorithms for the Sage
computer algebra system~\cite{sage}, can be found at
\begin{lyxlist}{ }
\item [{~}] \texttt{www.math.usm.edu/perry/Research/}
\end{lyxlist}
\noindent by appending to the above path one of the filenames
\begin{lyxlist}{ }
\item [{~}] \texttt{f5\_ex.lib}, \texttt{f5\_library.lib}, \texttt{f5\_library\_new.lib},
or\texttt{ basic\_sigbased\_gb.py}~.
\end{lyxlist}

\section*{Acknowledgments}

The authors wish to thank the Centre for Computer Algebra at Universit\"at
Kaiserslautern for their hospitality, encouragement, and assistance
with the \textsc{Singular} computer algebra system. Comments by Martin
Albrecht, Roger Dellaca, Lei Huang, and Alexey Zobnin helped improve
the paper.

\bibliographystyle{plain}
\bibliography{/home/perry/common/Research/researchbibliography}

\begin{thebibliography}{10}

\bibitem{AlbrechtPerryF45}
Martin Albrecht and John Perry.
\newblock F4/5.
\newblock Preprint, 2010.
\newblock Available online at arxiv.org/abs/1006.4933.

\bibitem{ArriPerry2009}
Alberto Arri and John Perry.
\newblock The {F}5 {C}riterion revised.
\newblock Submitted to the \emph{Journal of Symbolic Computation}, 2009,
  preprint online at \texttt{arxiv.org/abs/1012.3664}.

\bibitem{BWK93}
Thomas Becker, Volker Weispfenning, and Hans Kredel.
\newblock {\em Gr{\"o}bner Bases: a Computational Approach to Commutative
  Algebra}.
\newblock Springer-Verlag New York, Inc., New York, 1993.

\bibitem{BrickSlimgb_Journal}
Michael Brickenstein.
\newblock Slimgb: {G}r{\"o}bner bases with slim polynomials.
\newblock {\em Revista Matem{\'a}tica Complutense}, 23(2):453--466, 2009.

\bibitem{Buchberger65}
Bruno Buchberger.
\newblock {\em Ein {A}lgorithmus zum {A}uffinden der {B}asiselemente des
  {R}estklassenringes nach einem nulldimensionalem {P}olynomideal (An Algorithm
  for Finding the Basis Elements in the Residue Class Ring Modulo a Zero
  Dimensional Polynomial Ideal)}.
\newblock PhD thesis, Mathematical Institute, University of Innsbruck, Austria,
  1965.
\newblock {E}nglish translation published in the Journal of Symbolic
  Computation (2006) 475--511.

\bibitem{Buchberger85}
Bruno Buchberger.
\newblock Gr{\"o}bner-bases: An algorithmic method in polynomial ideal theory.
\newblock In N.~K. Bose, editor, {\em Multidimensional Systems Theory -
  Progress, Directions and Open Problems in Multidimensional Systems}, pages
  184--232, Dotrecht -- Boston -- Lancaster, 1985. Reidel Publishing Company.

\bibitem{Singular312}
W.~Decker, G.-M. Greuel, G.~Pfister, and H.~Sch{\"o}nemann.
\newblock {\sc Singular} 3-1-2.
\newblock {A Computer Algebra System for Polynomial Computations}, Centre for
  Computer Algebra, University of Kaiserslautern, 2010.
\newblock {\tt www.singular.uni-kl.de}.

\bibitem{EderPerry09}
Christian Eder and John Perry.
\newblock F5{C}: {A} variant of {F}aug{\`e}re's {F}5 algorithm with reduced
  {G}r{\"o}bner bases.
\newblock {\em Journal of Symbolic Computation}, 45(12):1442--1458, 2010.

\bibitem{Fau99}
Jean-Charles Faug{\`e}re.
\newblock A new efficient algorithm for computing {G}r{\"o}bner bases ({F}4).
\newblock {\em Journal of Pure and Applied Algebra}, 139(1--3):61--88, June
  1999.

\bibitem{Fau02Corrected}
Jean-Charles Faug{\`e}re.
\newblock A new efficient algorithm for computing {G}r{\"o}bner bases without
  reduction to zero {F}5.
\newblock In {\em International Symposium on Symbolic and Algebraic Computation
  Symposium - ISSAC 2002, Villeneuve d'Ascq, France}, pages 75--82, Jul 2002.
\newblock Revised version downloaded from
  \texttt{fgbrs.lip6.fr/jcf/Publications/index.html}.

\bibitem{GGV10}
Shuhong Gao, Yinhua Guan, and Frank Volny.
\newblock A new incremental algorithm for computing {G}roebner bases.
\newblock In {\em Proceedings of the 2010 International Symposium on Symbolic
  and Algebraic Computation}. ACM Press, 2010.

\bibitem{GM88}
Rudiger Gebauer and Hans M{\"o}ller.
\newblock On an installation of {B}uchberger's algorithm.
\newblock {\em Journal of Symbolic Computation}, 6:275--286, 1988.

\bibitem{KR00}
Martin Kreuzer and Lorenzo Robbiano.
\newblock {\em Computational Commutative Algebra {I}}.
\newblock Springer-Verlag, Heidelberg, 2000.

\bibitem{MMT92}
Hans M{\"o}ller, Ferdinando Mora, and Carlo Traverso.
\newblock Gr{\"o}bner bases computation using syzygies.
\newblock In P.~S. Wang, editor, {\em Proceedings of the 1992 International
  Symposium on Symbolic and Algebraic Computation}, pages 320--328. Association
  for Computing Machinery, ACM Press, July 1992.

\bibitem{sage}
William Stein.
\newblock {\em {Sage}: {O}pen {S}ource {M}athematical {S}oftware ({V}ersion
  3.1.1)}.
\newblock The Sage~Group, 2008.
\newblock {\tt www.sagemath.org}.

\bibitem{Zobnin_F5Journal}
Alexey Zobnin.
\newblock Generalization of the {F}5 algorithm for calculating {G}r{\"o}bner
  bases for polynomial ideals.
\newblock {\em Programming and Computer Software}, 36(2):75--82, 2010.

\end{thebibliography}

\end{document}